\documentclass[12pt]{amsart}
\usepackage{ amsmath, amsthm, amsfonts, amssymb, color}
 \usepackage{mathrsfs}
\usepackage{amsfonts, amsmath}
 \usepackage{amsmath,amstext,amsthm,amssymb,amsxtra}
 \usepackage{txfonts}
 \usepackage[colorlinks, citecolor=blue,hypertexnames=false]{hyperref}
 \allowdisplaybreaks
 \usepackage{pgf,tikz}
 \usepackage{fancybox}
\usepackage{tikz}
\usepackage{graphicx}
\usepackage{graphics}
\usepackage{esint}

\begin{document}
\baselineskip 17pt
\hfuzz=6pt

\newtheorem{theorem}{Theorem}[section]
\newtheorem{proposition}[theorem]{Proposition}
\newtheorem{cor}[theorem]{Corollary}
\newtheorem{lemma}[theorem]{Lemma}
\newtheorem{definition}[theorem]{Definition}
\newtheorem{example}[theorem]{Example}
\newtheorem{remark}[theorem]{Remark}
\newcommand{\ra}{\rightarrow}
\renewcommand{\theequation}
{\thesection.\arabic{equation}}
\newcommand{\ccc}{{\mathcal C}}
\newcommand{\one}{1\hspace{-4.5pt}1}

\newtheorem*{TheoremA}{Theorem A}

\newtheorem*{TheoremB}{Theorem B}

\newcommand{\vmo}{{\rm VMO}}
\newcommand{\BMO}{{\rm BMO}}
\newcommand{\VMO}{{\rm VMO}}
\newcommand{\La}{\Lambda}

\newcommand\CC{\mathbb{C}}
\newcommand\NN{\mathbb{N}}
\newcommand\ZZ{\mathbb{Z}}

\renewcommand\Re{\operatorname{Re}}
\renewcommand\Im{\operatorname{Im}}

\newcommand{\supp}{{\rm supp}{\hspace{.05cm}}}
 
\newcommand{\mar}[1]{{\marginpar{\sffamily{\scriptsize
        #1}}}}

	\makeatletter
	\newcommand{\rmnum}[1]{\romannumeral #1}
	\newcommand{\Rmnum}[1]{\expandafter\@slowromancap\romannumeral #1@}
	\makeatother

   \newcommand {\Rn}{{\mathbb R}^n}

\newcommand\wrt{\,{\rm d}}

\title[Weak type $(1,1)$ bounds for Riesz transforms]
{Weak type $(1,1)$ bounds for Riesz transforms for elliptic operators in non-divergence form
  }

\author{Liang Song and Huohao Zhang}

\address{Liang Song, School of Mathematics,
	Sun Yat-sen University,
	Guangzhou, 510275,
	P.R.~China}
\email{songl@mail.sysu.edu.cn}

\address{Huohao Zhang, School of Mathematics,
	Sun Yat-sen University,
	Guangzhou, 510275,
	P.R.~China}
\email{zhanghh58@mail2.sysu.edu.cn}

\date{\today}


\begin{abstract}
	{
Let $L=-\sum_{i,j=1}^n a_{ij}D_iD_j$ be the elliptic operator in non-divergence form with smooth real  coefficients satisfying uniformly elliptic condition. Let $W$ be the global 
 nonnegative adjoint solution. If $W\in A_2$,   we prove that the Riesz transforms $\nabla L^{-\frac{1}{2}}$ is of weak type $(1,1)$ with respect to the measure $W(x)dx$. This, together with $L^2_W$ boundedness of Riesz transforms  \cite{EHH}, implies that the Riesz transforms are bounded in $L^p_W$  for $1<p<2$. Our results are applicable to the case of real coefficients having sufficiently small BMO norm.
}

\end{abstract}

\maketitle


\section{Introduction}
\setcounter{equation}{0}

The main purpose of  this paper is to study the weak type $(1,1)$ bounds for  Riesz transforms associated to  
second-order elliptic operator in non-divergence form on $\Rn$.
More precisely, consider a  elliptic operator in non-divergence form  
\begin{equation}\label{e1.2}
     Lu=-\sum_{i,j=1}^n a_{ij} D_i D_ju.
\end{equation}
Assume that $A=(a_{ij}(x))$ is a matrix of real and measurable coefficients. Without loss of generality, we  take $A$ to be symmetric.  We also  assume the following uniformly elliptic condition, i.e., 
there exists a constant $ \lambda>0$ such that 
\begin{equation}\label{elliptic}
	A(x)\,\xi\cdot \xi \geq \lambda |\xi|^2 \ \ \ \text{ and  } \ \ \  |A(x)\,\xi\cdot\zeta| \leq \lambda^{-1}|\xi||\zeta|
\end{equation}
  for all $\xi,\zeta\in\mathbb{R}^n$  and a.e. $x\in \mathbb{R}^n$.  

Next, let us recall the definition of solutions of the adjoint equation (see \cite{E2}). A function $u\in L^1_{loc}(\Rn)$ is said to be a solution of the adjoint equation $L^*u=0$ if  for all $\phi\in C^\infty_c(\Rn)$ there holds
$$
\int_{\Rn} u(x)L\phi(x) \,dx=0,
$$
where $L^{\ast}$ is the adjoint operator of $L$, defined by
\begin{equation}\label{e1.3}
	L^{\ast} u=-\sum_{i,j=1}^n  D_iD_j (a_{ij}u). 
\end{equation}

   Let L and $L^\ast$ be as above. Escauriaza \cite{E2}  showed a very important property as follows.   
  There exists a non-negative solution $W$ of the adjoint equation $L^{\ast} W=0$ in ${\mathbb R}^n$, satisfying $W(B_1)=|B_1|$, 
which we call the {\it global non-negative adjoint solution}. Moreover, $W$ satisfies a reverse H\"older property with exponent $n/(n-1)$ so $W\in A_{\infty}$.  If  the coefficients of $L$ are smooth, or even belong to VMO, then $W$ (with the stated normalization) is unique.

It has been pointed out  in \cite{EHH} that $L$, as a linear unbounded operator on $L^2_W$,  is closed, sectorial and m-accretive. Particularly, the numerical
ranges $\{(Lu,u)\in \mathbb{C}: \|u\|_{L^2_W}=1\}\subseteq S_{\omega}:=\{z\in \mathbb{C}: |\arg z|\leq \omega\}\cup \{0\}$, with $0<\omega<\pi/2$.
Hence $L$ has an m-accretive square root (\cite[Theorem 3.35, p.281]{K}).  Also, $L$ generates a holomorphic semigroup $e^{-zL}$, $0\leq |Arg(z)|<\pi/2-\omega$ (\cite[Theorem 1.24, p.492]{K}). 

What is the exact domain of the square root of $L$? It is an important question. Recently,  Escauriaza, Hidalgo-Palencia and Hofmann  \cite{EHH}    answered  the  square root problem for non-divergence second order elliptic operators  in  the case of real  smooth coefficients having sufficiently small BMO norm (see Theorem \ref{kato} below). Theorem \ref{kato}  implies  the domain of $\sqrt{L}$  coincides the Sobolev space $H^1_W({\mathbb R}^n):=\{ u\in L^2_W({\mathbb R}^n): \ \nabla u\in 
 L^2_W({\mathbb R}^n)\}$. It is well known that the  square root problem  for divergence form complex-elliptic opetators in $\Rn$ was a long-standing open problem, which posed by Kato \cite{Ka} and refined by ${\rm M^c}$Intosh\cite{Mc82,Mc84}. It has been  solved by Auscher, Hofmann, Lacey,  ${\rm M^c}$Intosh and
Tchamitchian in \cite{AHLMcT}.

\begin{theorem}[\cite{EHH}] \label{kato}
	Let $L$ be a second-order elliptic operator in non-divergence form with smooth real coefficients 
satisfying \eqref{elliptic}, 
 	and let $W$ be the associated global non-negative adjoint solution $L^{\ast}W=0$.  If $W \in A_2$, then we have
	\begin{equation}\label{kato-inequality}
\big\|\sqrt{L}f\big\|_{L^2_W} \approx \big\|\nabla f\big\|_{L^2_W}\approx\big\|\sqrt{\tilde{L}}f\big\|_{L^2_W}, 
	\end{equation}
	where the implicit constants depend only on $n, \lambda$ and 
	$[W]_{A_2}$. 
	\end{theorem}
Let us give some remarks. (i)
The weighted Hilbert space $L^2_W(\Rn):=L^2(\Rn, W(x)dx)$ is more natural in many ways than unweighted space $L^2(\Rn)$ in the non-divergence case.  (ii)  It follows from \cite{E2} that  $W\in A_\infty$, and then $W\in A_p$ for some $p$ depending on dimension and ellipticity, but $p$  may be strictly greater than $2$. If, additionally,  the coefficients $\{a_{ij}(x)\}$ have sufficiently small BMO norm, then $W\in A_2$, and \eqref{kato-inequality} does  apply in that setting. 
(iii) The operator $\tilde{L}$ in Theorem \ref{kato} is the normalized adjoint of $L$, which is defined, at least for smooth coefficients, by the formula
$\tilde{L}u:=\frac{1}{W}L^*(uW)$. (iv) In the case $n=1$, the  square root problem for non-divergence second order elliptic operators  had been treated 
in \cite{KMe}. Indeed,   authors in \cite{KMe} treated the more general class of operators of the form $L=-aDbD$, where $D$  denotes the ordinary differentiation operator on $\mathbb R$, and $a,b$
are arbitrary bounded accretive complex-valued functions on $\mathbb R$.

 It follows from Theorem \ref{kato} that the Riesz transforms $\nabla L^{-\frac{1}{2}}$ are bounded on $L^2_W(\Rn)$. A natural question is to study the action of the Riesz transforms  on $L^p_W(\Rn)$ or other endpoint spaces. In this paper, we prove the weak type $(1,1)$ boundedness and  so   $L^p$ ($1<p<2$) boundedness of Riesz transforms $\nabla L^{-\frac{1}{2}}$.  The following theorem is our main result.

  \begin{theorem}\label{th1.1}
 	Let $L$ be a second-order elliptic operator in non-divergence form with smooth real coefficients 
satisfying \eqref{elliptic}.
 	Let $W$ be the associated global non-negative adjoint solution $L^{\ast}W=0$. If $W\in A_2$, then
	the operator  
 	$\nabla L^{-\frac{1}{2}}$ is
 	 of weak type $(1,1)$ with respect to $W(x)dx$ such that 
 	$$
 	W(\{|\nabla L^{-\frac{1}{2}} f(x)|>\alpha\}) 
  \leq \frac{C}{\alpha} \int_{\mathbb{R}^n}|f(x)| W(x) \,dx,
 	$$
    where $C$ depends only on $n,\lambda$ and $[W]_{A_2}$.
 Hence by interpolation,  the operator  
 $\nabla L^{-\frac{1}{2}}$   is  bounded on $L^p_W ({\mathbb R}^n)$ for $1<p\leq 2$. 
 \end{theorem}
 
 The proof of Theorem \ref{th1.1} is based on the approach, that used by Coulhon and  Duong \cite{CD} as treated Riesz transforms associated to Laplace-Beltrami operator on a complete Riemannian manifold satisfying the doubling
volume property. Related ideas also appeared earlier in  Duong and ${\rm M^c}$Intosh \cite{DMc}.   The key step is to  prove   one appropriate  estimate for the $L^2_W$ norm of $|\nabla_x \Gamma_{s^2}(x,y)|e^{2\gamma\frac{|x-y|^2}{s^2}}$, where $\Gamma_{s^2}$ is the kernel of heat semigroup $e^{-s^2L}$ (see Proposition \ref{est1} below). To get it, we will use the Gaussian bounds for kernels of the operators $e^{-t^2L}$ and $t^2Le^{-t^2L}$ (Lemmas \ref{heat kernel1} and  \ref{The absolute}), and  a key equality (\cite[(1.4)]{EHH}):
\begin{equation}\label{eq2}
		\int_{\Rn} \sum_{i,j=1}^n a_{ij} D_i f D_j f  W\,dx = \int_{\Rn} f L f W\,dx,  \quad \mbox{if} \ f\in {\mathcal D}(L).
  \end{equation}
 Here ${\mathcal D}(L):=\left\{f\in L^2_W(\mathbb{R}^n): \ Lf \in  L^2_W(\mathbb{R}^n)\right\}.
$  With Proposition \ref{est1} at hand, we then follow the argument of \cite{DMc,CD} to decompose the Riesz transforms $\nabla L^{-1/2}$ into two parts:
$\nabla L^{-1/2}e^{-r_j^2L}$ and $\nabla L^{-1/2}(I-e^{-r_j^2L})$ when we deal with the action of  $\nabla L^{-1/2}$ on   $b_j$, where $\{b_j\}$ is the bad part of the Calder\'on--Zygmund decomposition. Proposition \ref{est1} will play a crucial role in the estimate of $\nabla L^{-1/2}(I-e^{-r_j^2L})$ (see Section 4 below).

We also study the action of the Riesz transforms  on  Hardy space associated with  $L$. It is known that
  $L$ has a bounded $H_\infty$-calculus on $L^2_W(\Rn)$(\cite{EHH,DY02}). The  semigroup
$\{e^{-tL}\}_{t>0}$ satisfies the Davies--Gaffney condition (\cite{EHH}). With these properties at hand, the theory of Hardy spaces associated with $L$, $\mathcal{H}_{L,W}^1(\Rn)$, has been contained in \cite{DLi}, where includes  the molecular decomposition,  the square function characterization,
duality, and interpolation. 
We mention that the theory of Hardy space associated to operators has attracted a lot of attention in
the last two decades and has been a very active research topic in harmonic analysis.  We refer the reader to \cite{ADM,DY05,AMR,HM,HLMMY,DLi,SY} and the references therein.

By combining the properties of $L$ with the theory of Hardy space \cite{DLi}, we get the following result.

\begin{theorem}\label{th1.2}        Let $L, W$ be as Theorem \ref{th1.1}. Then there exists a constant $C>0$, such that  for any $f\in \mathcal{H}_{L,W}^1(\Rn)$,
\begin{align*}
\left\|\nabla L^{-\frac{1}{2}}f\right\|_{L^1_W(\Rn)}\leq C\left\|f\right\|_{\mathcal{H}_{L,W}^1(\Rn)}.
\end{align*}
\end{theorem}

The paper is organized as follows. In Section 2, we give some basic results including the properties of Muckenhoupt weights and equations in non-divergence form. In Section 3, we give a  proof of Proposition \ref{est1}, which plays a key role in the proof of Theorem \ref{th1.1}.
With Proposition \ref{est1} at hand, we follow the approach of \cite{CD,DMc} to  prove our main result, Theorem \ref{th1.1},  in Section 4. In Section 5, we recall the definition of $\mathcal{H}_{L,W}^1(\Rn)$, and give 
the proof of Theorem~\ref{th1.2}.

Throughout this paper, we shall  denote $B(x,r)$ or $B_r(x)$ the ball with the center  $x$ and the radius $r$, and $w(E):=\int_E w(x)\,dx$ for any set $E\subset \Rn$ and any weight $w$. The letter “c” and “C” will denote (possibly different) constants that are independent of the essential variables.

\bigskip


\section{Notation and preliminaries}
\setcounter{equation}{0}

\medskip
First, we recall some basic definitions and properties of Muckenhoupt weights. For  details we refer to \cite{Gr,Stein}. 
 \begin{definition}
	Let $1<p<\infty$. A weight $w$ is said to be of class Muckenhoupt $A_p$ if $w(x)>0$, a.e. $x\in {\mathbb R}^n$ and  
	\begin{equation*}
[w]_{A_p}:=\sup\limits_B\left(\fint_B w(x)\,dx\right)\left(\fint_{B} w^{-\frac{1}{p-1}} \,dx\right)^{p-1}<\infty,
	\end{equation*}
	where the supremum is taken over  all the balls $B\subset\mathbb{R}^n$. Also, define 
	$A_\infty:=\cup_{1<p<\infty} A_p$.
\end{definition}
The following facts are well known. 

(i) $A_1(\Rn)\subseteq A_p(\Rn)\subseteq A_q(\Rn)$ for $1\leq p\leq q\leq \infty$.

(ii) if $w\in A_p(\Rn)$, $1<p<\infty$, then there exists $1<q<p$ such that $w\in A_q(\Rn)$.

(iii) $A_p$ weights are doubling. Precisely, for all $a>1$ and all balls $B\subset \Rn$, there holds
\begin{align}\label{doubling property}
 w(a B)\leq a^{np}[w]_{A_p}w(B).   
\end{align}

(iv) $A_\infty$ property is equivalent to the Reverse H\"older property, i.e.,   there is an exponent $r>1$ and a uniform constant $C_r>0$ such that for each ball $B$ in $\mathbb{R}^n$,
	\[
	\left(\fint_B w^r \,dx\right)^{\frac{1}{r}}\leq C_r \fint_B w \,dx.
	\leqno(RH_r)
   \]

\medskip

\bigskip


Next, we state  the Gaussian  bounds for kernels of the operators $e^{-t^2L}$ and $t^2Le^{-t^2 L}$,  which plays an important role in our proof.

\begin{lemma}\label{heat kernel1}{\rm \cite[Theorem 1.2]{E2}}\	The kernel of $e^{-t^2 L}$, $\Gamma_{t^2}(x, y)$,  satisfies the Gaussian bounds,		\begin{eqnarray}\label{Gaussian_upper_bound}
	\Gamma_{t^2}(x, y) \leq C \min \left\{\frac{1}{W\left(B_t(x)\right)}, \frac{1}{W\left(B_t(y)\right)}\right\} e^{-c \frac{|x-y|^2}{t^2}} W(y) \text {, }
	\end{eqnarray}
	and
	\begin{eqnarray}\label{Gaussian_lower_bound}
		\frac{1}{C}\max \left\{\frac{1}{W\left(B_t(x)\right)}, \frac{1}{W\left(B_t(y)\right)}\right\} e^{-\frac{|x-y|^2}{c t^2}} W(y) \leq  \Gamma_{t^2}(x, y),
	\end{eqnarray}
	where constants $c, C$ depend on $n$ and $\lambda$.
\end{lemma}
\begin{lemma}\label{The absolute} {\rm (\cite[Remark 2.6]{EHH})}
The absolute value of the kernel of the operator $t^2Le^{-t^2 L}$  satisfies the upper bound \eqref{Gaussian_upper_bound}.
\end{lemma}

We end this section with a useful inequality.

\begin{lemma}
Let $w \in A_p$, for some $1\leq p<\infty$. There exists a  constant $\tilde{C}>0$,  such that for all $y\in \Rn$ and all $s>0$ 
\begin{align}\label{ee3.10}
\frac{1}{w(B_s(y))} \int_{\Rn} e^{-c\frac{|x-y|^2}{s^2}} {w(x)} \, dx   \leq \tilde{C},
\end{align}
where $\tilde{C}$ depends only on $c,n,p,[w]_{A_p}$.
\end{lemma}
\begin{proof}
\begin{align*}
     &\frac{1}{{w(B_s(y))}}\int_{\mathbb{R}^n} e^{-c\frac{|x-y|^2}{s^2}} {w(x)} \, dx \\
     &\leq \frac{1}{w(B_s(y))}\left(
		\int_{B_s(y)}  e^{-c\frac{|x-y|^2}{s^2}} w(x)\,dx+ 
	  \sum_{k=1}^\infty \int_{B_{2^{k+1}s}(y)\backslash B_{2^{k}s}(y)}  e^{-c\frac{|x-y|^2}{s^2}} w(x)\,dx
	  \right)\\
      &\leq  1+\sum_{k=1}^\infty e^{-c4^k}\frac{w(B_{2^{k+1}s}(y))}{w(B_s(y))}
    \end{align*}
By using {doubling property}, we have 
      $$
      {w(B_{2^{k+1}s}(y))}{}\leq 2^{(k+1)np}[w]_{A_p} w(B_s(y)). 
      $$
    Thus, 
   $$ 
   \frac{1}{{w(B_s(y))}}\int_{\mathbb{R}^n} e^{-c\frac{|x-y|^2}{s^2}} {w(x)} \, dx\leq 1+\sum_{k=1}^\infty e^{-c4^k}2^{(k+1)np}[w]_{A_p} \leq \tilde{C}(n,p,c,[w]_{A_p}).
 $$

\end{proof}

%


\bigskip

\section{Some estimates about the gradient of heat kernels}
\setcounter{equation}{0}

The  aim of this section is to prove the following proposition, which plays a crucial role  in the proof of Theorem~\ref{th1.1}.

\begin{proposition}\label{est1}
	There exists a constant $\gamma>0$, depends only on $n$ and $\lambda$,  such that for $a.e.\ y\in \mathbb{R}^n$, 
  $$
  \int_{\mathbb{R}^n}|\nabla_x \Gamma_{s^2}(x,y)|^2 e^{2\gamma\frac{|x-y|^2}{s^2}}W(x)\, dx \leq C \frac{(W(y))^2}{ s^2W(B_s(y))},
  $$
where the  constant $C$ depends only on $n, \lambda$ and $[W]_{A_2}$. 
 \end{proposition}

\begin{proof}  
  
\medskip

Take $\varphi\in C_c^{\infty}(B(0,1))$ such that  $0\leq \varphi\leq 1$ and $\varphi=1$  on $B(0,1/2)$.
For any $R>0$, denote $\varphi_R(x):=\varphi(\frac{x}{R})$. To prove Proposition  \ref{est1}, it suffices to prove that 
 there exist  constants $\gamma>0$ and $C>0$, depend only on $n$ and $\lambda$,  such that for any $R>1$,
    \begin{equation}\label{IR}
	\int_{\mathbb{R}^n}\left|\nabla_x \Gamma_{s^2}(x,y)\right|^2\left[\varphi_R(x-y)e^{\gamma\frac{|x-y|^2}{s^2}}\right]^2 W(x)\, dx \leq C \ \frac{(W(y))^2}{ W(B_s(y))}
	\left(\frac{1}{s^2}+\frac{1}{R^2}\right),	
    \end{equation}
for $a.e.\ y\in \mathbb{R}^n$. In fact, since the constant $C$  above is independent of $R$, then Proposition \ref{est1} follows from \eqref{IR}  by letting $R\rightarrow\infty$.

Let us prove \eqref{IR}.
 Since $W\in L^1_{loc}(\Rn)$, then $W(y)<\infty$ for  a.e. \,$y\in \Rn$. Fix $y\in \mathbb{R}^n$ such that $W(y)<\infty$. It is easy to see that 
	\begin{align}\label{e3.1}
	\int_{\mathbb{R}^n}|\Gamma_{s^2}(x,y)|^2 W(x)\,dx \leq C \frac{W(y)^2}{W(B_s(y))},
	\end{align}
    and 
 \begin{align}\label{e3.2}
\int_{\mathbb{R}^n}|L\Gamma_{s^2}(\cdot,y)(x)|^2 W(x)\,dx \leq C \frac{W(y)^2}{s^4W(B_s(y))}.
	\end{align}
   In fact, one can apply \eqref{Gaussian_upper_bound}   to obtain
	\begin{align*}
	 &\int_{\mathbb{R}^n}|\Gamma_{s^2}(x,y)|^2 W(x)\,dx \leq C \int_{\mathbb{R}^n}\left|  \frac{1}{W\left(B_s(y)\right)} e^{-c \frac{|x-y|^2}{s^2}} W(y)\right|^2W(x)\,dx \\
     &\leq C\frac{W(y)^2}{W(B_s(y))} \left( \frac{1}{W(B_s(y))}\int_{\Rn} e^{-2c \frac{|x-y|^2}{s^2}} W(x)\,dx  \right),
    \end{align*}
which, together with \eqref{ee3.10}, yields \eqref{e3.1}.   \eqref{e3.2} follows by  a similar argument and   Lemma \ref{The absolute}.

 By using (\ref{elliptic}),  one can write
$$
\int_{\Rn} |\nabla_x \Gamma_{s^2}(x,y)|^2 W(x)\, dx
	\leq \lambda^{-1} \int_{\Rn} A\nabla_x \Gamma_{s^2}(x,y)\cdot \nabla_x \Gamma_{s^2}(x,y)  W(x) \, dx. 
$$
Applying \eqref{eq2}, \eqref{e3.1}, \eqref{e3.2} and the fact $\Gamma_{s^2}(\cdot,y)\in {\mathcal D}(L)$,  we  obtain 
	\begin{align}\label{e3.3}
	 &\int_{\Rn} A\nabla_x \Gamma_{s^2}(x,y)\cdot \nabla_x \Gamma_{s^2}(x,y)  W(x) \, dx
		= \int_{\Rn} L \Gamma_{s^2}(\cdot,y)(x)  \Gamma_{s^2}(x,y)  W(x)\, dx  \nonumber\\
	&\leq \left(\int_{\mathbb{R}^n}|L\Gamma_{s^2}(\cdot,y)(x)|^2 W(x)\, dx\right)^{1/2} \left(\int_{\mathbb{R}^n}|\Gamma_{s^2}(x,y)|^2 W(x)\, dx\right)^{1/2}\\
	&\leq C\frac{W(y)^2}{s^2W(B_s(y))}.  \nonumber  	
    \end{align}

     
Denote $u(x):=\Gamma_{s^2}(x,y)$ and $v(x):=\varphi_R(x-y)e^{\gamma\frac{|x-y|^2}{s^2}}$, where $\gamma$   is a  small positive number to be determined later. Since  $u\in H^2_W$ and $v\in C_c^\infty$, then  $uv\in H^2_W$.  Then we can apply 
(\ref{eq2}) and the identity $L(uv)=vLu+uLv-2\sum_{i,j=1}^na_{ij}D_i uD_j v$
to obtain
  \begin{equation} \label{eq1}
   \begin{aligned}
	&\int_{\mathbb{R}^n} A\nabla_x (uv) \cdot \nabla_x (uv) W(x)\,dx=\int_{\mathbb{R}^n} L (uv) uvW(x)\,dx\\
	=&\int_{\mathbb{R}^n} Lu \ uv^2 W(x)\, dx
	+\int_{\mathbb{R}^n} Lv \ u^2vW(x)\, dx 
	-2\int_{\mathbb{R}^n} \sum_{i,j=1}^na_{ij}D_iuD_jv  uvW(x)\, dx\\
	=&\int_{\mathbb{R}^n} Lu \, uv^2 W(x)\,dx
	+\int_{\mathbb{R}^n} Lv \, u^2vW(x)\,dx 
	-2\int_{\mathbb{R}^n} A\nabla_x u \cdot \nabla_x v uvW(x)\,dx.
   \end{aligned}
  \end{equation}
  Note that  \eqref{elliptic} implies that
   \begin{equation}\label{ineq1}
   \begin{aligned}
   &\int_{\mathbb{R}^n} A\nabla_x (uv) \cdot \nabla_x (uv) W(x)\,dx \\
   &\geq \lambda \int_{\mathbb{R}^n} |\nabla_x (uv)|^2  W(x)\,dx \\
   &= \lambda  \int_{\mathbb{R}^n} |\nabla_x u|^2v^2  W(x)\,dx
   + \lambda\int_{\mathbb{R}^n} |u|^2|\nabla_x v|^2 W(x)\,dx
   +2\lambda\int_{\mathbb{R}^n} (\nabla_x u\cdot \nabla_x v)  uvW(x)\,dx \\
   &\geq \lambda \int_{\mathbb{R}^n} |\nabla_x u|^2 v^2  W(x)\,dx 
   +2\lambda\int_{\mathbb{R}^n} (\nabla_x u\cdot \nabla_x v)  uvW(x)\,dx.
  \end{aligned}
  \end{equation}

\noindent Then, we combine (\ref{eq1}) and (\ref{ineq1}) to obtain
\begin{align}\label{e3.4}
&\int_{\mathbb{R}^n} |\nabla_x u|^2 v^2  W(x)\,dx \nonumber\\
&\leq
  C\left\{\int_{\mathbb{R}^n} |Lu| |u|v^2 W(x)\,dx+\int_{\mathbb{R}^n} |\nabla_x u||\nabla_x v|  |uv|W(x)\,dx
	+\int_{\mathbb{R}^n} |Lv| u^2|v|W(x)\,dx\right\}\\
    &=:\Rmnum{1}+\Rmnum{2}+\Rmnum{3}. \nonumber
\end{align}
Consider the term  \Rmnum{1}.  By Lemmas \ref{heat kernel1} and \ref{The absolute}, we have 
	\begin{align*}
	  \Rmnum{1} &\leq C \frac{1}{s^2} \left(\int_{\Rn} |s^2L \Gamma_{s^2}(\cdot,y)(x)|^2e^{\gamma\frac{|x-y|^2}{s^2}}W(x) \, dx \right)^{\frac{1}{2}}
	  \left(\int_{\Rn} | \Gamma_{s^2}(\cdot,y)(x)|^2e^{\gamma\frac{|x-y|^2}{s^2}}W(x) \, dx \right)^{\frac{1}{2}}\\
      &\leq C\frac{1}{s^2} \int_{\Rn} \frac{1}{W\left(B_s(y)\right)^2}  e^{-(2c-\gamma) \frac{|x-y|^2}{s^2}} {W(y)}^2 W(x) \, dx.
    \end{align*}  
Whenever $\gamma\leq c$, one can use \eqref{ee3.10} to obtain
\begin{align}\label{e3.9}
I\leq C\frac{1}{s^2} \int_{\Rn}   e^{-c \frac{|x-y|^2}{s^2}} W(x) \, dx  \frac{{W(y)}^2}{W\left(B_s(y)\right)^2}\leq C{s^{-2}}\frac{{W(y)}^2}{W\left(B_s(y)\right)}.
\end{align}

\medskip

For the term $II$, we use the Cauchy-Schwartz inequality to obtain 
	\begin{align}\label{e3.7}
	\Rmnum{2}&\leq \frac{1}{2}\int_{\mathbb{R}^n} |\nabla_x u|^2 v^2  W(x)\,dx+2C \int_{\mathbb{R}^n} |\nabla_x v|^2 u^2  W(x)\,dx.
    \end{align}
Observe that $\int_{\mathbb{R}^n} |\nabla_x u|^2 v^2  W(x)\,dx<\infty$ by \eqref{e3.3}. Then the first term of the RHS of (\ref{e3.7}) can be controlled by the LHS of (\ref{e3.4}). So to estimate the term  $II$, it suffices to prove the following
\begin{align}\label{e-II}
\int_{\mathbb{R}^n} |\nabla_x v|^2 u^2  W(x)\,dx\leq C \frac{(W(y))^2}{W(B_s(y))} \left(\frac{1}{R^2}+\frac{1}{s^2}\right).
\end{align}
In fact, one can obtain that
	 \begin{align}\label{e3.5}
	 \left|\nabla_xv\right|^2 &\leq \left(|(\nabla_x\varphi_R)(x-y)| e^{\gamma\frac{|x-y|^2}{s^2}}+\varphi_R(x-y)\frac{2\gamma |x-y|}{s^2}e^{\gamma\frac{|x-y|^2}{s^2}} \right)^2\nonumber\\
	 &\leq 2\left(
	 |(\nabla_x\varphi_R)(x-y)|^2 e^{2\gamma\frac{|x-y|^2}{s^2}}+|\varphi_R(x-y)|^2 \frac{4\gamma^2|x-y|^2}{s^4}e^{2\gamma\frac{|x-y|^2}{s^2}}\right)\\
     &\leq C e^{3\gamma\frac{|x-y|^2}{s^2}}\left(\frac{1}{R^2}+\frac{1}{s^2}\right),\nonumber
	\end{align}
where in the last inequality we have used the fact $s^{-2}{|x-y|^2}e^{2\gamma\frac{|x-y|^2}{s^2}}\leq C e^{3\gamma\frac{|x-y|^2}{s^2}}$ and $|\nabla_x\varphi_R|\leq CR^{-1}$.
 \eqref{e3.5}, together with \eqref{Gaussian_upper_bound} and \eqref{ee3.10} , yields
	\begin{align*}
	\int_{\mathbb{R}^n} |\nabla_x v|^2 u^2  W(x)\,dx &\leq C  \frac{(W(y))^2}{W(B_s(y))^2} \left(\frac{1}{R^2}+\frac{1}{s^2}\right)\int_{\mathbb{R}^n} 
	  {W(x)}e^{(3\gamma-2c)\frac{|x-y|^2}{s^2}}
	  \, dx\\
      &\leq C \frac{(W(y))^2}{W(B_s(y))} \left(\frac{1}{R^2}+\frac{1}{s^2}\right),
       \end{align*}
  whenever $\gamma\leq c/3$.  This completes the proof of \eqref{e-II}.

\medskip

It remains to estimate the term \Rmnum{3}.   We first obtain
	$$
	\begin{aligned}
	  \left|\nabla_x^2v\right| &\leq C \left\{\frac{1}{s^2} \left(4\gamma^2\frac{|x-y|^2}{s^2}+2\gamma\right) e^{\gamma\frac{|x-y|^2}{s^2}}
	  +\frac{2\gamma|x-y|}{s^2}\frac{1}{R} e^{\gamma \frac{|x-y|^2}{s^2}}
	  +\frac{1}{R^2} e^{\gamma\frac{|x-y|^2}{s^2}}\right\}\\
	  &\leq C \left\{\frac{1}{R^2}e^{2\gamma \frac{|x-y|^2}{s^2}} +\frac{1}{s^2}\left(\frac{|x-y|^2}{s^2} +1\right)e^{\gamma \frac{|x-y|^2}{s^2}}\right\},
	\end{aligned}
	$$
where the constant $C$ depends on $\gamma$ and $\varphi$.
	Also, by the assumption $\left\|a_{ij}\right\|_{L^\infty}\leq \lambda^{-1}$, we have 
	$$
	\left|Lv\right|\leq C \left\{\frac{1}{R^2}e^{2\gamma \frac{|x-y|^2}{s^2}} +\frac{1}{s^2}\left(\frac{|x-y|^2}{s^2} +1\right)e^{\gamma \frac{|x-y|^2}{s^2}}\right\}.
	$$
    Then a similar argument to that of \eqref{e-II} gives
    \begin{align}\label{e3.8}
    \Rmnum{3}\leq C \frac{(W(y))^2}{W(B_s(y))} \left(\frac{1}{R^2}+\frac{1}{s^2}\right),
    \end{align}
whenever $\gamma\leq c/3$.  Letting $\gamma=c/3$, we combine \eqref{e3.9},\eqref{e-II} and \eqref{e3.8} to   complete of the proof of  \eqref{IR}.
\end{proof}

\medskip

 The following  proposition is a direct corollary  of  Proposition \ref{est1}.

\begin{proposition}\label{prop3.1}
	There exist constants $\beta, C>0$, depend only on $n$ and $\lambda$, such that the following estimate holds 
	\begin{equation}\label{2.4.1}
		\int_{|x-y|>t} \left|\nabla_x \Gamma_{s^2}(x,y)\right| W(x)\,dx \leq C  s^{-1} e^{-\beta  t^2/s^2 }  W(y),
	\end{equation}
    for all $s>0,t>0$, and  a.e.   $y\in \Rn$.
\end{proposition}
\begin{proof}	 
By applying  Lemma \ref{est1} and the Cauchy-Schwartz inequality, one has
		\begin{align*}\label{ee3.1}
	 \int_{|x-y|>t} |\nabla_x \Gamma_{s^2}(x,y)| W(x)\, dx&\leq    \left(\int_{\mathbb{R}^n}|\nabla_x \Gamma_{s^2}(x,y)|^2e^{2\gamma\frac{|x-y|^2}{s^2}}W(x)\, dx\right)^{\frac{1}{2}}               
	   \left(\int_{|x-y|>t} e^{-2\gamma \frac{|x-y|^2}{s^2}} W(x)\, dx\right)^{\frac{1}{2}}\\
       &\leq  C \frac{W(y)}{ sW(B_s(y))^{1/2}}e^{-\frac{\gamma}{2} {t^2}/{s^2}}\left(\int_{\Rn} e^{-\gamma \frac{|x-y|^2}{s^2}} W(x)\, dx\right)^{\frac{1}{2}}\\
       &\leq C s^{-1} e^{-{\gamma t^2}/{(2s^2)} }  W(y),
	\end{align*}
where in the last inequality we have used \eqref{ee3.10}.
	    \end{proof}

 \medskip


\section{Proof of Theorem~\ref{th1.1}}
\setcounter{equation}{0}

The aim of the section is to prove Theorem~\ref{th1.1}. With Propositions \ref{est1} and \ref{prop3.1} at our disposal, our argument  will  follow the approach of \cite{CD}, as treated Riesz transforms associated to Laplace-Beltrami operator on a complete Riemannian manifold satisfying the doubling
volume property and an optimal on-diagonal
heat kernel estimate.
Let us first state the classical Calder\'on-Zygmund decomposition, whose proof can be found in \cite{GR,Stein}.

\begin{lemma}\label{e4.1}
Let $f\in L^1(Wdx)$, and $\alpha>0$. Then there exists a decomposition of $f$,  $f=g+b$, with $b=\sum_j b_j$, and a sequence of balls $B_j:=B(x_j,r_j)$, so that
\begin{align*}
 &(i) \quad |g(x)|\leq C\alpha, \ \  \mbox{ for a.e. }x\in \Rn; \\
 &(ii) \quad   \supp b_j\subset B_j \ \  {\rm and} \ \ \int_{B_j} |b_j(x)| W(x)\,dx\leq C\alpha W(B_j); \\
&(iii)  \quad  \sum_{j} W(B_j) \leq \frac{C}{\alpha} 
  \int_{\mathbb{R}^n} |f(x)| W(x)\,dx; \\ 
&(iv)	 \quad  \sum_{j} \chi_{B_j} \leq C;\\
&(v) \quad \alpha \leq \frac{1}{W(B_j)}\int_{B_j} |f(x)| W(x)\,dx\leq C\alpha. 
\end{align*}

\end{lemma}

\begin{proof}[Proof of Theorem~\ref{th1.1}]  
Denote $T:=\nabla L^{-\frac{1}{2}}$.  It suffices to prove  that  there exists  a constant $C>0$, for any $\alpha>0$, and all $f\in L^1_W(\Rn)$, there holds
\begin{equation}\label{weak11}
	W\left(\big\{x: |Tf(x)|>\alpha\big\}\right)\leq \frac{C}{\alpha} \int_{\mathbb{R}^n}|f(x)| W(x)\, dx. 
\end{equation}
Since $L^1_W(\Rn)\cap L^2_W(\Rn)$ is dense in $L^1_W(\Rn)$, we may suppose $f\in L^1_W(\Rn)\cap L^2_W(\Rn)$ without loss of generality.  

Applying Lemma \ref{e4.1}, one has 
$$
W\left(\{x:| Tf(x)|>\alpha\}\right)\leq W\left(\{x:|Tg(x)|>\alpha/3\}\right)+W\left(\{x:|Tb(x)|>2\alpha/3\}\right)=:\Rmnum{1}+\Rmnum{2}.
$$
 Since   the $L^2_W$ boundedness of $T$ has been proved (Theorem \ref{kato}),  then we use Lemma 4.1 to obtain 
$$
I
\leq C \alpha^{-2}\left\|g\right\|_{L^2_W}^2 \leq C \alpha^{-1}\left\|g\right\|_{L^1_W} 
\leq C \alpha^{-1}\left\|f\right\|_{L^1_W}. 
$$

\medskip

Consider the term $\Rmnum{2}$.  Denote $E^*:=\bigcup_j B_j^*$, where $B_j^*=B(x_j, 3r_j)$.    One can estimate that
 $$
 \begin{aligned}
 \Rmnum{2}&\leq W(E^*)+W\left\{x\in (E^*)^c:\left|\sum_j\left(Te^{-r_j^2L}\right)b_j(x)\right|>\alpha/3\right\}\\
 &\hskip 1.8cm+ W\left\{x\in (E^*)^c: \left|\sum_j D_j b_j(x)\right|>\alpha/3\right\}
 \\
 &=:\Rmnum{2}_1+\Rmnum{2}_2+\Rmnum{2}_3,
\end{aligned}
 $$
 where $D_j:=T(I-e^{-r_j^2L})$.
   By (iii) of Lemma \ref{e4.1} and  the doubling property  \eqref{Doubling}, 
 we have
 $$
 \Rmnum{2}_1\leq C \frac{1}{\alpha} \int_{\mathbb{R}^n}|f(x)|W(x)\, dx.
 $$

 Consider  the term  $\Rmnum{2}_2$. We will show the following inequality:
 \begin{equation} \label{eq4.6}
 \left\|\sum_{j} e^{-r_j^2 L}b_j\right\|_{L^2_W}^2 \leq C \alpha \left\|f\right\|_{L^1_W}.
\end{equation}
If \eqref{eq4.6} has been proved, then we use Chebyshev’s inequality and Theorem \ref{kato} to obtain
$$
II_2\leq C\frac{1}{\alpha^2} \left\|T\left(\sum_{j} e^{-r_j^2 L}b_j\right)\right\|_{L^2_W}^2\leq  \frac{C}{\alpha} \left\|f\right\|_{L^1_W},
$$
which is our desired conclusion.

Let us prove \eqref{eq4.6}. Using (\ref{Gaussian_upper_bound}), we have that for any $x\in (E^*)^c$,
\begin{align*}
  |e^{-r^2_j L}b_j(x)|&\leq \int \Gamma_{r_j^2}(x,y)|b_j(y)|\, dy\\
&\leq C  \int_{B_j}
\min \left\{\frac{1}{W\left(B_{r_j}(x)\right)}, \frac{1}{W\left(B_{r_j}(y)\right)}\right\} e^{-c \frac{|x-y|^2}{r_j^2}} |b_j(y)| W(y) \, dy\\
&\leq Ce^{-\frac{c}{4} \frac{|x-x_j|^2}{r_j^2}}\frac{1}{W(B_{r_j}(x))} \left\|b_j\right\|_{L^1_W},
\end{align*}
where in the last inequality we have used the fact $|x-x_j|/2\leq |x-y|\leq 2|x-x_j| $ when $x\notin B(x_j,3r_j)$ and $y\in B_j$. Next, we  apply (ii) of Lemma \ref{e4.1} to have
\begin{align}\label{ee4.1}
  |e^{-r^2_j L}b_j(x)| &\leq C \alpha e^{-\frac{c}{4} \frac{|x-x_j|^2}{r_j^2}}\frac{W(B_j)}{W(B_{r_j}(x))} 
\leq C \alpha \int_{B_j} \frac{1}{W(B_{r_j}(x))} e^{-\frac{c}{4}\frac{|x-x_j|^2}{r_j^2}}W(y)\,dy \nonumber\\
&\leq C \alpha \int_{B_j} \frac{1}{W(B_{r_j}(x))} e^{-\frac{c}{16}\frac{|x-y|^2}{r_j^2}}W(y)\,dy  \nonumber\\
&\leq C \alpha \int_{B_j} \frac{1}{W(B_{r_j}(y))} \left(\frac{|x-y|}{r_j}\right)^{2n} e^{-\frac{c}{16}\frac{|x-y|^2}{r_j^2}}W(y)\,dy\\
&\leq C \alpha \int_{B_j} \frac{1}{W(B_{r_j}(y))} e^{-\frac{c}{17}\frac{|x-y|^2}{r_j^2}}W(y)\,dy, \nonumber
\end{align}
where in the last inequality above we have used the doubling property of $W$ (see Section 2.1):
$$W(B_{r_j}(y))\leq W(B_{|x-y|+r_j}(x)) \leq [W]_{A_2}   \left(1+\frac{|x-y|}{r_j}\right)^{2n} W(B_{r_j}(x)).$$

\noindent With \eqref{ee4.1} at our disposal,  in order to prove \eqref{eq4.6},  it suffices to show that
\begin{equation}\label{esrm2.1}
  \left\|\sum_j \int_{\mathbb{R}^n}   \frac{1}{W\left(B_{r_j}(y)\right)}
     e^{-\frac{c}{17}\frac{|\cdot-y|^2}{r_j^2}} \chi_{B_j}(y) W(y)\,dy
     \right\|_{L^2_W} \leq C \left\|\sum_j \chi_{B_j}\right\|_{L^2_W}.
\end{equation}
In fact, if \eqref{esrm2.1} has been proved, then we apply (iii) and  (iv) of Lemma \ref{e4.1} to obtain
\begin{align*}
\left\|\sum_{j} e^{-r_j^2 L}b_j\right\|_{L^2_W}^2 
&\leq C\alpha^2\int \left(\sum_j \chi_{B_j}\right)^2 Wdx \leq C \alpha^2\sum_{j} \int_{B_j} W\,dx\\
&=C \alpha^2\sum_{j}W(B_j)\leq C \alpha \left\|f\right\|_{L^1_W}.
\end{align*}

Let us prove (\ref{esrm2.1}).  By duality, it suffices to estimate the term  
$$
{\mathcal I}:=\left|\int_{\mathbb{R}^n} \sum_j \int_{\mathbb{R}^n} \frac{1}{W\left(B_{r_j}(y)\right)}
e^{-\frac{c}{17}\frac{|x-y|^2}{r_j^2}} \chi_{B_j}(y) u(x)W(y)\,dy W(x)\,dx \right|,
$$ 
where  $u$ is any function satisfying $\left\|u\right\|_{L^2_W}=1$.
Note that 
$$
\begin{aligned}
  &\frac{1}{W(B(y,r_j))}\int_{\mathbb{R}^n} e^{-\frac{c}{17}\frac{|x-y|^2}{r_j^2}} |u(x)| W(x)\, dx \\
  =&\frac{1}{W(B(y,r_j))} \left(\int_{B(y,r_j)} 
   +\sum_{k=0}^\infty \int_{B(y,2^{k+1}r_j)\setminus{B(y,2^{k}r_j) }}\right)e^{-\frac{c}{17}\frac{|x-y|^2}{r_j^2}} |u(x)| W(x)\,dx\\
  \leq & M_W u(y)+C\sum_{k=0}^\infty \frac{W(y,2^{k+1}r_j)}{W(y,r_j)}e^{-\frac{c}{17}4^k}M_W u(y)  \leq C  M_W u(y),
\end{aligned}
$$
where $M_W$ is the Hardy--Littlewood maximal operator with respect to $W(x) dx$:
\begin{equation*}
  M_W f(y) :=\sup_{y\in B} \frac{1}{W(B)} \int_{B}|f(x)|W(x)\,dx.
\end{equation*}
Using the Cauchy--Schwartz inequality and the $L^2_W$ boundedness of  $M_W$, we have
$$
{\mathcal I}\leq \int_{\mathbb{R}^n} M_Wu(y) \sum_j \chi_{B_j}(y) W(y)\, dy
\leq \big\|M_Wu\big\|_{L^2_W} \Big\|\sum_j \chi_{B_j}\Big\|_{L^2_W}
\leq C \Big\|\sum_j \chi_{B_j}\Big\|_{L^2_W},
$$ 
which yields \eqref{esrm2.1}.

\bigskip

Finally, it remains to estimate the term $\Rmnum{2}_3$. One can write 
$$
L^{-\frac{1}{2}}=\int_0^\infty e^{-sL} \, \frac{ds}{\sqrt{s}},
$$
then
$$
L^{-\frac{1}{2}}(I-e^{-tL})=\int_0^\infty \left(\frac{1}{\sqrt{s}}-\frac{\chi_{\{s>t\}}(s)}{\sqrt{s-t}}\right) e^{-sL}\,ds,
$$
where $\chi_{\{s>t\}}$ is the characteristic function of $(t,+\infty)$.

Let  $K_t(x,y)$ be the kernel of $T(I-e^{-tL})=\nabla L^{-\frac{1}{2}}(I-e^{-tL})$,
and denote 
\begin{equation}
  g_t(s):=\frac{1}{\sqrt{s}}-\frac{\chi_{\{s>t\}}}{\sqrt{s-t}}.
\end{equation}
Then we have 
$$
K_t(x,y)=\int_0^\infty g_t(s)\nabla_x \Gamma_s (x,y) \, ds.
$$
Thus by Proposition  \ref{prop3.1},
$$
\begin{aligned}
   \int_{|x-y|\geq \sqrt{t}} |K_t(x,y)| W(x)\, dx 
   &\leq C \int_0^\infty |g_{t}(s)|e^{-\beta t/s } \frac{1}{\sqrt{s}} W(y) \, ds \\
   &=C\left(\int_0^{t} \frac{e^{-\beta t/s }}{s} \, ds+
   \int_t^\infty \left(\frac{1}{\sqrt{s-t}}-\frac{1}{\sqrt{s}}\right)e^{-\beta t/s } \,\frac{ds}{\sqrt{s}}  
   \right) W(y)\\
   &=:C\left(\Rmnum{2}_{31}+\Rmnum{2}_{32}\right) W(y).
\end{aligned}
$$

Let $u=s/t$, then
$$
\Rmnum{2}_{31}=\int_0^1 e^{-\beta/u} \frac{1}{u} \,du,
$$
which is a finite number that depends only on $\beta$ (thus, depends on $n$ and $\lambda$).
Next, consider $\Rmnum{2}_{32}$. Letting $u=s-t$, we have 
\begin{align*}
\Rmnum{2}_{32}&\leq\int_0^\infty \left(\frac{1}{\sqrt{u}}-\frac{1}{\sqrt{u+t}} \right) \frac{1}{\sqrt{u+t}} \, du 
=\int_0^\infty \left( \frac{1}{\sqrt{u(u+t)}}-\frac{1}{u+t} \right)\, du\\
&=\int_0^\infty  \left(\frac{1}{\sqrt{vt(vt+t)}}-\frac{1}{vt+t} \right)t\, dv
=\int_0^\infty \left(\frac{1}{\sqrt{v(v+1)}}-\frac{1}{v+1} \right)\, dv <\infty,
\end{align*}
where in the second equality above we used the change of variables $u=vt$. 

Thus,
\begin{align}\label{Key}
\int_{|x-y|\geq \sqrt{t}} |K_t(x,y)| W(x)\, dx \leq CW(y).
\end{align}
Then 
\begin{equation}\label{ee4.2}
\begin{aligned}
\int_{(3B_j)^c}|T(I-e^{-r_j^2L})b_j|W(x)\, dx
&\leq \int_{(3B_j)^c}\left|\int_{B_j}K_{r_j^2}(x,y)b_j(y) \,dy\right|W(x)\, dx \\
&\leq \int_{\mathbb{R}^n} \int_{|x-y|>r_j} |K_{r_j^2}(x,y)|W(x)\, dx|b_j(y)|\, dy \\
&\leq C \int_{\mathbb{R}^n} |b_j(y)|W(y)\, dy.
\end{aligned}
\end{equation}
Thus, we use \eqref{ee4.2} and Chebyshev’s inequality to obtain
\begin{align*}
\Rmnum{2}_3&\leq \frac{C}{\alpha}\sum_j\int_{(3B_j)^c}|T(I-e^{-r_j^2L})b_j|W(x)\, dx\\
&\leq \frac{C}{\alpha}\sum_j\left\|b_j\right\|_{L^1_W}\leq C\sum_j W(B_j)\leq \frac{C}{\alpha}\|f\|_{L^1_W}.
\end{align*}

This completes the proof of (\ref{weak11}).
\end{proof}

\medskip

\section{Hardy space boundedness of Riesz transforms}

Let $L=-\sum_{i,j=1}^n a_{ij} D_i D_j$. The goal of  this section is to investigate the boundedness of Riesz transforms $\nabla L^{-1/2}$ on Hardy spaces associated to $L$.

As mentioned in Introduction,   $L$ is  a linear operator of type $\omega$ on $L^2_W(\Rn)$ with $\omega<\pi/2$, hence $L$ generates a holomorphic semigroup $e^{-zL}$, $0\leq |Arg(z)|<\pi/2-\omega$.   Also the following properties hold:

 {\bf (H1)}  $L$ has a bounded $H_\infty$-calculus on $L^2(Wdx)$ (\cite[Remark 1.13]{EHH}). That is, there exists $c_{\mu,2}>0$ such that  for $b\in H_\infty(S_\mu^0)$ and $f\in L^2_W(\Rn)$,
$$
\|b(L)f\|_{L^2_W}\leq c_{\mu,2}\|b\|_\infty\|f\|_{L^2_W}.
$$

 {\bf(H2)} The  semigroup
$\{e^{-tL}\}_{t>0}$ satisfies the Davies-Gaffney condition (\cite[Lemma 2.14]{EHH}).
That is, there exist constants $C$, $c>0$ such that for any open subsets
$U_1,\,U_2\subset \Rn$,
\begin{equation*}
\int_{U_2}\left| e^{-t^2L}f(x) \right|^2W(x)\,dx\leq C e^{-c\frac{{\rm dist}(U_1,U_2)^2}{t^2}} \int_{U_1}\left| f(x)\right|^2W(x)\,dx, \qquad \mbox{ if } \supp f\subset U_1
\end{equation*}

\noindent 
where ${\rm dist}(U_1,U_2):=\inf_{x\in U_1, y\in U_2} d(x,y)$.

For details of the theory of functional calculus, We refer the reader to see \cite{K,Mc,AlDM,Auscher}. 
 \medskip

Next, let us recall  the definitions of  Hardy spaces associated to  $L$ (see \cite{ADM,AMR,HLMMY,DLi}).

Define 
\begin{equation*}
S_{L}f(x):=\left(\iint_{\Gamma(x)}
|t^2Le^{-t^2L} f(y)|^2 {W(y)\over W(B_t(x))}\,{dydt\over t}\right)^{1/2},
\quad x\in \Rn,
\end{equation*}
where $\Gamma(x):=\{(y,t)\in \Rn\times (0,\infty): |x-y|<t\}$.
 Define 
$\mathcal{H}^2_{L,W}(\Rn):= \overline{\mathcal{R}(L)}$,
that is, the closure of the range of $L$ in $L^2_W(\Rn)$.   Then $L^2_W(\Rn)$
is the orthogonal sum of $\mathcal{H}^2_{W}(\Rn)$ and the null space $\mathcal{N}(\tilde{L})$, where $\tilde{L}$ is   adjoint operator of $L$ in $L^2_W(\Rn)$, as defined in Introduction. 

For $1\leq p<\infty$, define
the Hardy space  ${\mathcal H}^p_{L,W}$ associated to $L$  as the 
completion of
$\{f\in \mathcal{H}_{L,W}^2(\Rn):\,\| S_Lf\|_{L^p_W(\Rn)}<\infty\}$,  with respect to the norm $\|S_Lf\|_{L^p_W(\Rn)}$,  and
\begin{equation*}
\|f\|_{{\mathcal H}^p_{L,W}}:=\| S_Lf\|_{L^p_W(\Rn)},\quad f\in \mathcal{H}^2_{L,W}(\Rn).
\end{equation*}
We note that by Theorem \ref{kato}, it follows that 
$\mathcal{N}(\sqrt{\tilde{L}})=\{0\}$, and thus  $\mathcal{N}({\tilde{L}})=\{0\}$. 
As a result, $\mathcal{H}^2_{L,W}(\Rn)=L^2_W(\Rn)$. By a similar argument of \cite{ADM}, we have 
$\mathcal{H}^p_{L,W}(\Rn)=L^p_W(\Rn)$, for $1<p<\infty$.

Now we give the proof of Theorem \ref{th1.2}.
\smallskip

\noindent {\it Proof of Theorem \ref{th1.2}.}
It is a direct corollary of  \cite[Theorem 4.9]{DLi}, in which  a similar argument to \cite{HM} was used. In fact, by Theorem \ref{kato}, $\nabla L^{-\frac{1}{2}}$ is bounded on $L^2_W(\Rn)$. It  is also known from \cite[Lemma 2.14]{EHH} that  the family of operators $\sqrt{t}\nabla e^{-tL}$ satisfies the Davies--Gaffney estimates. The conditions of  \cite[Theorem 4.9]{DLi} are satisfied,  thus  we finishs the proof of   Theorem \ref{th1.2}.

\begin{remark}
 
(i) Theorem \ref{th1.2}, together with the interpolation theorem of \cite[Theorem 4.7]{DLi},  implies that  $\|\nabla L^{-\frac{1}{2}}f\|_{\mathcal{H}_{L,W}^p(\Rn)}\leq C_p \|f\|_{\mathcal{H}_{L,W}^p(\Rn)}$,  for all $1<p<2$.    Since $\mathcal{H}_{L,W}^p(\Rn)=L^p_W(\Rn), (1<p<2)$,  we show  $L^p_W(\Rn)$ boundedness of the Riesz transform $\nabla L^{-1/2}$  for $1<p<2$ again.

(ii) Up to now, we do not know whether $\|\nabla L^{-\frac{1}{2}}\|_{L_W^p(\Rn)}\leq C_p \|f\|_{L_W^p(\Rn)}$ hold for some $p>2$.  It is a  very interesting question. 
\end{remark}

 \bigskip

\noindent
{\bf Acknowledgements}: The authors would like to thank Prof. Lixin Yan for helpful discussion and useful suggestions. The authors   were supported  by National Key R$\&$D Program of China 2022YFA1005700.
  L. Song was  supported by  NNSF of China (No. 12471097).

\end{document}